\documentclass{amsart}
 \usepackage[T2A]{fontenc}
\usepackage[margin=1in]{geometry}
\usepackage{dsfont}
\usepackage{amsmath, amsthm, amssymb, mathrsfs}
\usepackage[abbrev]{amsrefs}
\usepackage[all]{xy}
\usepackage{graphicx}
\usepackage{extarrows}
\usepackage{hyperref}
\usepackage{xcolor}
\usepackage{tikz-cd}
\newtheorem{thm}{Theorem}
\newtheorem{cor}{Corollary}
\newtheorem{lem}{Lemma}

\renewcommand{\geq}{\geqslant}

\usepackage[most]{tcolorbox}

\author{Dmitry Gayfulin}
\address{\parbox{\linewidth} {National Research University Higher School of Economics,\\ 11 Pokrovsky Bulvar, Moscow, 109028, Russia.}}

\email{gamak.57.msk@gmail.com (corresponding author)}
\author{Sergei Pitcyn}
\address{\parbox{\linewidth} {Moscow Center of Fundamental and Applied Mathematics}}
\email{picyn98@mail.ru}

\title{A note on exact approximations}

\begin{document}

\maketitle
\begin{abstract}
Based on M. Hall's theorem, we prove a simple result about real numbers that admit exact approximations by rationals.
\vskip+0.3cm
{\bf MSC}: 11J04
\end{abstract}

\vskip+1.5cm

\section{Exact approximations}

It is very well known due to Hurwitz's theorem that for every real irrational number $\alpha$ there exist
infinitely many rational numbers $\frac{p}{q}$ satisfying the inequality
\begin{equation}
\label{Hurwitz_ineq}
\left|\alpha - \frac{p}{q}\right| < \frac{1}{\sqrt{5}q^2}.
\end{equation}
Moreover, $\sqrt{5}$ in the denominator of \eqref{Hurwitz_ineq} cannot be replaced by any greater constant. If we put $\alpha=\frac{\sqrt{5}+1}{2}$, for any $\varepsilon>0$ the inequality
\begin{equation}
\label{Hurwitz_ineq_rev}
\left|\alpha - \frac{p}{q}\right|<\frac{1}{(\sqrt{5}+\varepsilon)q^2}
\end{equation}
has at most finitely many solutions in $\frac{p}{q}\in\mathbb{Q}$.

Properties \eqref{Hurwitz_ineq} and \eqref{Hurwitz_ineq_rev} can be generalized in the following way. Given $\alpha\in\mathbb{R}\setminus\mathbb{Q}$, suppose that there exists $\gamma$ such that the
inequality 
\begin{equation}
\label{1e}
\left|\alpha - \frac{p}{q}\right| < \frac{1}{\gamma q^2}
\end{equation}
has infinitely many solutions in rationals $\frac{p}{q}$, while a stronger inequality
\begin{equation}
\label{2e}
\left|\alpha - \frac{p}{q}\right| < \frac{1}{(\gamma+\varepsilon) q ^2}
\end{equation}
for any $\varepsilon>0$ has at most finitely many solutions.
In this situation, we will say that $\gamma$ satisfies the \textbf{exact} condition. One can easily see that the fact that $\gamma$ satisfies the exact condition for some irrational $\alpha$ implies that
\begin{equation}
\label{Lag_const}
\gamma=\left(\liminf_{p\in \mathbb{Z}, q \in \mathbb{Z}_+} q\left|q\alpha - p\right|\right)^{-1}.
\end{equation}

The right-hand side of \eqref{Lag_const} is usually called the Lagrange constant of $\alpha$. This quantity is denoted by $\lambda(\alpha)$. The set $\mathbb{L}$ of all possible values of $\lambda(\alpha)$ when $\alpha$ runs through the set of all irrational numbers is known as {\it the Lagrange spectrum}.  A famous theorem by Markoff (see 
%\cite{M} for the original proof and 
books \cites{Cas,CusFla} for detailed exposition and discussions) claims that the set ${\mathbb{L} \cap \left[0, 3\right)}$ is discrete and countable. It consists of increasing numbers 
$$
L_1   = {\sqrt{5}} <
L_2  = {\sqrt{8}} <
L_3  = \ {\sqrt{221}/5} <
\cdots <
L_j = L(m_j) =  {\sqrt{9- \frac{4}{m_j^2}}} <
\cdots,
$$
where ${m_j,\,j=1,2,3,\ldots}$ run over the set of so-called {\it Markoff numbers} i.e. the positive integers $m$ such that there exist $1\le m_1,m_2\le m$ satisfying the Markoff equation
\begin{equation*}
m^2+m_1^2+m_2^2=3mm_1m_2.    
\end{equation*}
We recall that $\mathbb{L}$ is a closed set and therefore its complement is a countable union of open intervals $(a, b)$ such that $(a,b) \cap \mathbb{L} = \emptyset$, but $a$ and $b$ both lie in the Lagrange spectrum. Such intervals are called the \textit{maximal gaps}.

It is known that all the values of $\gamma$ from the discrete part of the Lagrange spectrum satisfy the exact condition
(see  \cite{Cas},  Chapter II and \cite{CusFla}). 
Of course, if
$\gamma\not\in\mathbb{L}$, this number does not satisfy the exact condition.
However, some years ago, the first-named author \cites{G1,G2} 
pointed out a mistake in Malyshev's paper \cite{maly} and showed
that there exist numbers $\gamma\in \mathbb{L}$ which do not satisfy the exact  condition. The real numbers $\gamma$ satisfying the exact condition are called \textit{admissible}. The following theorem was established in \cite{G1}. 

\vskip+0.3cm

\begin{thm}
\label{thm_g1}
The number $\gamma \in \mathbb{L}$ is not admissible if and only if $\gamma$ is the left endpoint of a maximal gap in the Lagrange spectrum and there is no quadratic irrationality $\alpha$ such that $\lambda(\alpha) = \gamma$.
\end{thm}

Particularly, it was shown in \cite{G1} that the number $[3;3,3, 2, 1, \overline{1,2}] + [0;2,1,\overline{1,2}]\in\mathbb{L}$ is not admissible.

\vskip+0.3cm

It follows from Hall's \cite{Hall} theorem that there exists a constant $\mu_0$ such that $[\mu_0,\infty)\in\mathbb{L}$. The half-line $[\mu_0,\infty)$ is usually called Hall's ray. The exact value of the origin of Hall's ray was found by Freiman, $\mu_0\approx 4.5278$. Taking into account Theorem \ref{thm_g1}, we obtain that every $\gamma\ge\mu_0$ is admissible (see also Theorem 2 from \cite{Bug2}).
The basic facts related to the results mentioned above and other properties of $\mathbb{L}$ can be found in \cite{CusFla}.

\section{Approximations with decreasing functions}

Instead of considering the inequalities of the form \eqref{1e} and \eqref{2e}, one can replace $\frac{1}{q^2}$ by a different function decaying to $0$. Particularly, consider two functions 
$ \Psi (q) >\Psi_1 (q) $ decreasing to zero as $ q\to \infty$. We are interested in existence of $\alpha\in \mathbb{R}\setminus \mathbb{Q}$ such that the inequality
\begin{equation}
\label{1ee}
\left|\alpha-\frac{p}{q}\right|<\Psi(q)
\end{equation}
has infinitely many solutions in rationals $\frac{p}{q}$, while a stronger inequality 
\begin{equation}
\label{2ee}
\left|\alpha-\frac{p}{q}\right|< \Psi_1(q)
\end{equation}
has at most finitely many solutions.

\vskip+0.3cm
In the last two decades, several papers devoted to different aspects of exact approximations have been published (see, for example  \cites{a,BW,Bug1,Bug2,BugMor,ggmm,mm,MoshPit}).
Particularly, in \cite{BW}, the following result was established. 

\begin{thm}
\label{thm_g2}
Suppose that the functions $\Psi(q)$ and $\omega(q)$ satisfy
\begin{equation}
\label{2eee}
\lim_{q\to\infty}\limits q^{2}\Psi(q)=\lim_{q\to\infty}\limits\omega(q)=0
\end{equation}
and $q\to q^{2}\Psi(q)$ is non-increasing. Furthermore, suppose that 
\begin{equation}
\label{2eeee}
\omega(q)\geq \frac{45}{\Psi(q)}\Psi\left(\frac{1}{2q\Psi(q)}\right)\,\,\,\,\,
\text{    for all}\,\, q\,\text{ sufficiently large and }\,
    \Psi_1(q)= \Psi(q)(1- \omega(q)).
\end{equation}
Then there exists an uncountable set $E$   such that for any $ \alpha \in E$ the inequality (\ref{1ee}) has infinitely many rational solutions, while the inequality (\ref{2ee})   has at most finitely many solutions.
\end{thm}
 
\vskip+0.3cm

Note that Theorem \ref{thm_g2} cannot be applied to the function $ \Psi(q) =\frac{1}{\gamma q^2}$ as the condition (\ref{2eee}) is not satisfied. Nevertheless, if we define the function 
$\omega(q)$ taking equality in \eqref{2eeee}, we get 
a constant function
 \begin{equation}
 \label{2y}
 \omega(q) =\frac{180}{\gamma^2}.
 \end{equation} 
 In the present paper, we consider just the case $ \Psi(q) =\frac{1}{\gamma q^2}$. In our results, the analogue of the function  $\omega(q)$ in \eqref{2eeee} and \eqref{2y} decays to zero faster than $\frac{1}{q^{2-\varepsilon}}$ for any $\varepsilon>0$, making the inequalities \eqref{1ee} and \eqref{2ee} in some sense best possible.
 Although our argument is very close to that of Theorem 2 from \cite{Bug2}, to the best
of the authors' knowledge, the following results have never been published before.
\vskip+0.3cm

\begin{thm}
\label{thm_oneside}
Given $\gamma>5$, let $\varpi: \mathbb N \longrightarrow \mathbb R_{+}$ be an arbitrary function increasing to infinity. Then there exists a number $\alpha$ such that the inequality
\begin{equation}
\label{8}
\left|\alpha - \frac{p}{q}\right| < \frac{1}{\gamma q^2}
\end{equation}
has infinitely many solutions in reduced fractions $\frac{p}{q}$, while the inequality
\begin{equation}
\label{8_conv}
\left|\alpha- \frac{p}{q}\right| < \frac{1}{\gamma q^2}\left(1 - \frac{\varpi(q)}{q^2}\right)
\end{equation}
has at most finitely many rational solutions.
\end{thm} 
\begin{thm}
\label{thm_twoside}
Let $\varpi: \mathbb N \longrightarrow \mathbb R_{+}$ be an arbitrary function increasing to infinity, and let $\gamma\ge4+[0;\overline{3,1}]+[0;2,\overline{1,3}]= 4.62202 \ldots$. Then there exists a number $\alpha$ such that the inequality
\begin{equation*}
\left|\alpha - \frac{p}{q}\right| < \frac{1}{\gamma q^2}\left(1 + \frac{\varpi(q)}{q^2}\right)
\end{equation*}
has infinitely many solutions in reduced fractions $\frac{p}{q}$, while the inequality
\begin{equation*}
\left|\alpha- \frac{p}{q}\right| < \frac{1}{\gamma q^2}\left(1 - \frac{\varpi(q)}{q^2}\right)
\end{equation*}
has at most finitely many rational solutions.
\end{thm}

\vskip+0.3cm

\textbf{Remark 1.}  In some sense, the result of Theorem \ref{thm_oneside} is optimal in order.  It is shown in \cite{MoshPit} 
(Example 1 to Theorem 2 from \cite{MoshPit}) that for  $\gamma \in \mathbb{Q}$  there exists
$\gamma_1>0$ such that there is no irrational $\alpha$ satisfying the following condition: the inequality (\ref{8}) has infinitely many solutions, while the inequality 
$$
\left|\alpha - \frac{p}{q}\right| < \frac{1}{\gamma q^2}\left(1 - \frac{\gamma_1}{q^2}\right)
$$
has at most finitely many rational solutions. A quite similar statement is also proven in \cite{BW} (see Theorem 1.6  from \cite{BW}).
\vskip+0.3cm
\textbf{Remark 2.} However, the lower bounds on $\gamma$ in Theorems \ref{thm_oneside} and \ref{thm_twoside} are not optimal. We conjecture that Theorem \ref{thm_twoside} holds for all elements of $\mathbb{L}$ and Theorem \ref{thm_oneside} holds for all $\gamma$ for which there exists $x$ such that $\lambda(x)=\gamma$ and the inequality \eqref{8} has infinitely many rational solutions.

\section{Notation, continued fractions and Hall-type theorems}
In our argument, we will exploit the properties of the decomposition of real numbers into ordinary continued fractions
\begin{equation}
\label{cf}
\alpha=[a_0;a_1, \ldots, a_n, \ldots], a_0\in\mathbb{Z}, a_{j\ge1}\in \mathbb{Z}_+.
\end{equation}
We denote the convergents to $\alpha$ as follows
\begin{equation}
\label{conv_alpha}
\frac{p_n(\alpha)}{q_n(\alpha)}:=[a_0;a_1,\ldots,a_n].    
\end{equation}
Sometimes we will use the notation $p_n(a_1,\ldots,a_n)$ and $q_n(a_1,\ldots,a_n)$ for the numerator and the denominator of the left-hand side of \eqref{conv_alpha} respectively.
The dependency of $p_n$ and $q_n$ on $\alpha$ or $a_1,\ldots,a_n$ will be omitted if it does not create ambiguity. 
We will frequently use the following notation: suppose that $\alpha$ has a decomposition \eqref{cf}, denote 
\begin{equation*}
\alpha_n^{*}:=[0;a_n,a_{n-1},\ldots,a_1]=\frac{q_{n-1}}{q_n}, \quad  \alpha_n=[a_n;a_{n+1},a_{n+2},\ldots].  
\end{equation*}
We also denote by $\lambda_n(\alpha)$ the following quantity
\begin{equation}
\label{lambda_def}
\lambda_n(\alpha):=[a_{n};a_{n+1},\ldots]+[0;a_{n-1},\ldots,a_1]=\alpha_{n} + \alpha_{n-1}^{*}.
\end{equation}
The following equality is known as Perron's formula:
\begin{equation}
\label{Perron}
\left|\alpha-\frac{p_n}{q_n}\right|=\frac{1}{\lambda_{n+1}(\alpha)q_n^2}.
\end{equation}
From \eqref{lambda_def} and \eqref{Perron} one can deduce that 
\begin{equation}
\label{Lag_eq}
\limsup\limits_{n\to\infty}\lambda_n(\alpha)=\lambda(\alpha).
\end{equation}

In our proofs, we rely on the properties of the sets of continued fractions with elements uniformly bounded from above. Denote
\begin{equation*}
F_k=\left\{[0;a_1,\ldots,a_n,\ldots]\ |\ a_i\le k\ \forall i\in\mathbb{Z}_{+}  \right\}.    
\end{equation*}
The following theorem was established by M. Hall \cite{Hall}.
\begin{thm}
\begin{equation}
\label{sum4}
F_4+F_4=[\sqrt{2}-1, 4\sqrt{2}-4]\approx [0.41421, 1.65685].
\end{equation}    
\end{thm}
As the length of the interval at the right-hand side of \eqref{sum4} is greater than one, one can easily deduce the following statement
\begin{cor}
\label{cor_repr6}
Every real number greater than $6$ can be represented as a sum
\begin{equation}
\label{repr6}
c+[0;b_1,\ldots]+[0;c_1,\ldots]
\end{equation}
where $c\ge 5$ is an integer and the partial quotients $b_i, c_i$ for $i=1,2,\ldots$ do not exceed $4$.
\end{cor}
Since the foundational paper \cite{Hall} by Hall was published, several generalizations of the property \eqref{sum4} have been found. We need two of them. The first one is due to Freiman and Judin \cite{Freiman_Judin}.
\begin{thm}
Any real number in the interval $[5-\sqrt{21}, \sqrt{21}-3]\approx [0.41742, 1.58257]$ can be written in the form $[0;b_1,\ldots]+[0;c_1,\ldots]$, where the partial quotients $b_i, c_i$ for $i=1,2,\ldots$ do not exceed $4$ and where no ordered pair $(b_i, b_{i+1})$ or $(c_i,c_{i+1})$ is ever equal to $(1,4)$ or $(2,4)$.
\end{thm}
\begin{cor}
\label{cor_repr5}
Any real number $\gamma\in[10-\sqrt{21}, 6]\approx[5.41742, 6]$ can be represented as a sum
\begin{equation}
\label{repr5}
\gamma=5+[0;b_1,\ldots]+[0;c_1,\ldots],
\end{equation}
where the partial quotients $b_i, c_i$ for $i=1,2,\ldots$ do not exceed $4$ and where no ordered pair $(b_i, b_{i+1})$ or $(c_i,c_{i+1})$ is ever equal to $(1,4)$ or $(2,4)$.
\end{cor}
The second result we will apply was established by Freiman \cite{Freiman} and Astels \cite{astels}.
\begin{thm}
The set $F_3+F_3$ contains the interval
\begin{equation*}
[[0;\overline{3,1}]+[0;2,\overline{1,3}], [0;\overline{1,3}]+[0;1,2,\overline{1,3}]]\approx[0.62202,1.52752].   
\end{equation*}
\end{thm}
\begin{cor}
\label{cor_repr4}
Any real number $\gamma\in\left[4+[0;\overline{3,1}]+[0;2,\overline{1,3}], 10-\sqrt{21}\right]\approx [4.62202, 5.41742]$ can be represented as a sum
\begin{equation}
\label{repr4}
\gamma=4+[0;b_1,\ldots]+[0;c_1,\ldots]
\end{equation}
where the partial quotients $b_i, c_i$ for $i=1,2,\ldots$ do not exceed $3$.
\end{cor}
We will denote the continued fractions $[0;b_1,\ldots]$ and $[0;c_1,\ldots]$ in the representations \eqref{repr6}, \eqref{repr5} and \eqref{repr4} by $\mu$ and $\nu$ respectively.

\section{Proofs}
We start with the proof of Theorem \ref{thm_oneside}.
\begin{proof}
First of all, note that one can consider the statement of Theorem \ref{thm_oneside} for $\frac{p}{q}=\frac{p_n(\alpha)}{q_n(\alpha)}$ only. Indeed, if $\frac{p}{q}$ is not a convergent to $\alpha$, then, by Legendre’s theorem, one has
\begin{equation*}
\left|\alpha-\frac{p}{q}\right|\ge\frac{1}{2q^2}>\frac{1}{\gamma q^2}
\end{equation*}
thus violating \eqref{8}. Hence, in order to prove Theorem \ref{thm_oneside}, it is sufficient to show that the inequality 
\begin{equation*}
%\label{8_adv}
\left|\alpha - \frac{p_n}{q_n}\right| < \frac{1}{\gamma q_n^2}
\end{equation*}
is satisfied for infinitely many $n$, while the inequality
\begin{equation*}
\left|\alpha- \frac{p_n}{q_n}\right|>\frac{1}{\gamma q_n^2}\left(1 - \frac{\varpi(q_n)}{q_n^2}\right)
\end{equation*}
is satisfied for all $n$ large enough. 
By Perron's formula \eqref{Perron}, this is equivalent to the fact that
\begin{equation}
\label{lambda_n_gt}
\lambda_n(\alpha)>\lambda(\alpha)
\end{equation}
for infinitely many $n$ and
\begin{equation*}
%\label{lambda_n_lt_0}
\lambda_{n+1}(\alpha)<\lambda(\alpha)\left(1-\frac{\varpi(q_n)}{q^2_n}\right)^{-1}
\end{equation*}
for all $n$ large enough.

Next, note that if the inequality \eqref{8_conv} is satisfied for some function $\varpi(n)$ then it holds for any function $\varpi'(n)\ge\varpi(n)$. Hence, without loss of generality we can assume that $\frac{\varpi(q_n)}{q_n^2}<\frac{1}{2}$ for all $n\ge 1$. As
\begin{equation*}
\lambda(\alpha)\left(1-\frac{\varpi(q_n)}{q^2_n}\right)^{-1}>\lambda(\alpha)\left(1+\frac{\varpi(q_n)}{q^2_n}\right)>
\lambda(\alpha)+\frac{\varpi(q_n)}{q^2_n},
\end{equation*}
it is sufficient to show that 
\begin{equation}
\label{lambda_n_lt}
\lambda_{n+1}(\alpha)<\lambda(\alpha)+\frac{\varpi(q_n)}{q^2_n}
\end{equation}
for all $n$ large enough.

Now we start the proof of Theorem \ref{thm_oneside}.
Suppose that $5<\gamma\le 10-\sqrt{21}$. Note that
\begin{equation*}
4+[0;4,\overline{1,3}]+[0;\overline{1,3}]=5.    
\end{equation*}
Thus, there exists $\varepsilon=\varepsilon(\gamma)>0$ and a constant $n_0\in\mathbb{N}$ such that if
\begin{equation}
\label{31_pattern_in_x}
\alpha=[a_0;a_1,\ldots,\underbrace{3,1,3,1,\ldots,3,1}_{n_0 \text{ pairs}},4^*,4,\underbrace{1,3,1,3,\ldots,1,3}_{n_0 \text{ pairs}},\ldots]    
\end{equation}
for some $\alpha\in\mathbb{R}\setminus\mathbb{Q}$, one has 
\begin{equation}
\label{lambda_01}
\max\left(\lambda_r(\alpha),\lambda_{r+1}(\alpha)\right)<\gamma-\varepsilon.
\end{equation}
The star in \eqref{31_pattern_in_x} denotes the $r$-th partial quotient. Denote  by $C$ the sequence
\begin{equation*}
C:=\underbrace{3,1,3,1,\ldots,3,1}_{n_0 \text{ pairs}},4,4,\underbrace{1,3,1,3,\ldots,1,3}_{n_0 \text{ pairs}}
\end{equation*}
As $5<\gamma\le 10-\sqrt{21}$, due to Corollary \ref{cor_repr4}, $\gamma$ can be represented as a sum \eqref{repr4}. Denote
\begin{equation}
\label{Bnm_def}
B^{n}_m:=b_m,b_{m-1},\ldots, b_1,4,c_1,c_2,\ldots,c_n.  
\end{equation}
\begin{lem}
\label{lem_k_i}
Let $n_k$ and $m_k$ be arbitrary sequences of odd positive integers growing to infinity. For any $5<\gamma\le 10-\sqrt{21}$ and any $\alpha$ of the form
\begin{equation}
\label{alpha_in_blocks}
\alpha:=[0;C,B^{n_1}_{m_1},C,B^{n_2}_{m_2},C,\ldots]
\end{equation}
one has $\lambda(\alpha)=\gamma$. Moreover, for any $5<\gamma\le 10-\sqrt{21}$ there exists a positive constant $\varepsilon=\varepsilon(\gamma)$ such that for any $\alpha$ of the form \eqref{alpha_in_blocks} and any $n\in\mathbb{N}$ such that $a_n(\alpha)$ is not $4$ in some block $B^{n_i}_{m_i}$, the inequality $\lambda_{n}(\alpha)<\gamma-\varepsilon$ holds.
\end{lem}
Given $\alpha$ in the form \eqref{alpha_in_blocks}, we will denote by $k_i$ the sequence of indices of $4$-s in the blocks  $B^{n_i}_{m_i}$.
\begin{proof}
Denote the partial quotients of $\alpha$ by $a_m$. One can easily see that
\begin{equation*}
\lim\limits_{i\to\infty}\lambda_{k_i}(\alpha)=\gamma.  
\end{equation*}
On the other hand, if $n\ne k_i$ and $a_n\le 3$, then $\lambda_{n}(\alpha)<5$. Finally, if $a_n=4$ and $a_n$ lies in some block $C$, due to \eqref{lambda_01}, we have $\lambda_{n}(\alpha)<\gamma-\varepsilon$. Thus we established that 
\begin{equation*}
\limsup\limits_{n\to\infty}\lambda_{n}(\alpha)=\lim\limits_{i\to\infty}\lambda_{k_i}(\alpha)
\end{equation*}
Applying \eqref{Lag_eq}, we obtain that $\lambda(\alpha)=\gamma$ thus finishing the proof.
\end{proof}
\begin{lem}
\label{leftineq}
Given $5<\gamma\le 10-\sqrt{21}$ and $\alpha$ of the form \eqref{alpha_in_blocks}. For each $i\in\mathbb{N}$ one has $\alpha_{k_i-1}^{*}>\mu$.
\end{lem}
\begin{proof}
The first $m_i$ partial quotients of $\alpha_{k_i-1}^{*}$ and $\mu$ coincide by definition of $\alpha$. Next,
\begin{equation*}
\begin{split}
\mu=[0;b_1,\ldots,b_{m_i},\ldots]\le& [0;b_1,\ldots,b_{m_i},3,1,\ldots,3,1,\overline{{\bf 3},1}]\\<&
[0;b_1,\ldots,b_{m_i},\underbrace{3,1,\ldots,3,1}_{n_0 \text{ pairs}},{\bf 4},\ldots]=\alpha_{k_i-1}^{*}.
\end{split}
\end{equation*}
The last inequality is true because $m_i$ is odd and therefore the first different partial quotient (highlighted in bold in the above inequality) has even index.
\end{proof}
\begin{cor}
\label{corol_m_n}
Fix $i\ge 1$. Suppose that the odd numbers $m_i, n_{i-1},m_{i-1},\ldots,n_1 ,m_1$ are already defined. There exists a number $n^{0}_i=n^{0}_i(m_i, n_{i-1},m_{i-1},\ldots,n_1 ,m_1)$ such that for any $n_i\ge n^{0}_i$ and any $\alpha$ of the form \eqref{alpha_in_blocks} one has
\begin{equation*}
\lambda_{k_i}(\alpha)>\gamma.    
\end{equation*}
\end{cor}
\begin{proof}
It follows from \eqref{lambda_def} that the condition of the corollary is equivalent to
\begin{equation*}
\alpha_{k_i-1}^{*}-\mu>\nu+4-\alpha_{k_i}.
\end{equation*}
Denote $\alpha_{k_i-1}^{*}-\mu=\varepsilon$ which is greater than zero by Lemma \ref{leftineq}. Note that $\alpha_{k_i}$ tends to $\nu + 4$ as $n_i$ goes to infinity, hence $\left|\nu+4-\alpha_{k_i}\right|<\varepsilon$ for all $n_i$ large enough. Thus there exists a number $n^0_i$ such that 
\begin{equation}
\label{col4_impl}
n_i\ge n^0_i\quad\implies\quad \left|\nu+4-\alpha_{k_i}\right|<\alpha_{k_i-1}^{*}-\mu.
\end{equation}
\end{proof}
Consider $\alpha$ of the form \eqref{alpha_in_blocks} such that $n_i\ge n^{0}_i(m_i,n_{i-1},m_{i-1},\ldots,n_1 ,m_1)$ for all $i=1,2,\ldots$. Combining Corollary \ref{corol_m_n} and Lemma \ref{lem_k_i}, we ensure that the inequality \eqref{lambda_n_gt} is satisfied for infinitely many $n$.  Due to Lemma \ref{lem_k_i}, to prove \eqref{lambda_n_lt} it is sufficient to show that 
\begin{equation}
\label{lambda_k_i_lt}
\lambda_{k_i}(\alpha)<\lambda(\alpha)+\frac{\varpi(q_{k_i-1})}{q^2_{k_i-1}}=\gamma+\frac{\varpi(q_{k_i-1})}{q^2_{k_i-1}}.
\end{equation}
for all $i$ large enough.

\begin{lem}
Given $\alpha$ of the form \eqref{alpha_in_blocks} suppose that $n_i\ge n^{0}_i(m_i,n_{i-1},m_{i-1},\ldots,n_1 ,m_1)$ for some $i\ge 1$. If 
\begin{equation}
\label{red_to_2_den}
\alpha^{*}_{k_i-1}-\mu<\frac{\varpi(q_{k_{i}-1})}{2q^2_{k_{i}-1}}
\end{equation}
then \eqref{lambda_k_i_lt} is satisfied.
\end{lem}
\begin{proof}
Recall that $4+\mu+\nu=\lambda(\alpha)=\gamma$. Thus, applying \eqref{col4_impl} and \eqref{red_to_2_den}, we get
\begin{equation}
\lambda_{k_i}(\alpha)-\lambda(\alpha)=\alpha_{k_i}-(4+\nu)+\alpha^{*}_{k_i-1}-\mu\le 2(\alpha^{*}_{k_i-1}-\mu)<\frac{\varpi(q_{k_{i}-1})}{q^2_{k_{i}-1}}.
\end{equation}
\end{proof}

To complete the proof of Theorem \ref{thm_oneside} for the case $5<\gamma\le 10-\sqrt{21}$ it remains to verify the inequality \eqref{red_to_2_den}.
We recall the following property of continuants (see \cite[Lemma 3.2]{ABD} for reference): for any two natural numbers $m<n$ one has
\begin{equation}
\label{cont2}
q_n(a_1,\ldots,a_n)\le 2 q_m(a_1,\ldots,a_m)q_{n-m}(a_{m+1},\ldots,a_n).    
\end{equation}
Recall also that
\begin{equation*}
\alpha=[0;a_1,\ldots,a_{k_i-m_i-1},a_{k_i-m_i}=b_{m_i},\ldots,a_{k_i-1}=b_1,a_{k_i}=4,\ldots].
\end{equation*}
Applying \eqref{cont2} to $\alpha$ with $n=k_i-1$ and $m=n-m_i$, one gets
\begin{equation}
\label{after_2_split}
q_{k_i-1}\le 2 q_{k_i-m_i-1}q_{m_i}(b_{m_i},\ldots,b_1)=2 q_{k_i-m_i-1}q_{m_i}(b_{1},\ldots,b_{m_i}),
\end{equation}
where the last equality follows from well-known symmetry of continuants (see again \cite[Lemma 3.2]{ABD}). Since the first $m_i$ partial quotients in continued fraction expansions of the numbers $\alpha^{*}_{k_i-1}$ and $\mu$ coincide and are equal to $b_1,\ldots, b_{m_i}$, a classical estimate (see \cite[Lemma 5.1]{AB}) yields
\begin{equation}
\label{alpha_star_minus_mu}
\left|\alpha_{k_i-1}^{*}-\mu\right|<\frac{1}{\left(q_{m_i}( b_{1},\ldots,b_{m_i})\right)^2}\le\frac{4q^2_{k_i-m_i-1}}{q^2_{k_i-1}}. 
\end{equation}

Since the quantity $q_{k_i-m_i-1}$ is independent of $m_i$, and the function $\varpi$ is increasing to infinity, we can choose $m_i$ large enough to satisfy the inequality $\varpi(q_{k_i-1})>8q^2_{k_i-m_i-1}$, which implies \eqref{red_to_2_den}. This completes the proof for the case $5<\gamma=\lambda(\alpha)\le 10-\sqrt{21}$. 

The remaining cases are considered by slightly modifying the original argument. Suppose that $10-\sqrt{21}<\gamma\le 6$. By Corollary \ref{cor_repr5}, $\gamma$ has the representation \eqref{repr5}
where the partial quotients $b_i, c_i$ for $i=1,2,\ldots$ do not exceed $4$ and where no ordered pair $(b_i, b_{i+1)}$ or $(c_i,c_{i+1})$ is ever equal to $(1,4)$ or $(2,4)$. We will construct $\alpha$ of the form \eqref{alpha_in_blocks}, but the definitions of the blocks $B^n_m$ and $C$ will be slightly different. Define
\begin{equation}
\begin{split}
\label{BC_2}
B^{n}_m:=b_m,b_{m-1},\ldots, b_1,5,c_1,c_2,\ldots,c_n, \quad C:=1,4,4,1.
\end{split}
\end{equation}
We formulate the following analogue of Lemma \ref{lem_k_i}.
\begin{lem}
\label{lem_k_i_2}
Let $n_k$ and $m_k$ be arbitrary sequences of {\bf even} positive integers growing to infinity. For any $10-\sqrt{21}<\gamma\le 6$ and any $\alpha$ of the form \eqref{alpha_in_blocks}, where $B^{n_i}_{m_i}$ and $C$ were defined in \eqref{BC_2},
one has $\lambda(\alpha)=\gamma$. Moreover, for any $10-\sqrt{21}<\gamma\le 6$ there exists a positive constant $\varepsilon=\varepsilon(\gamma)$ such that for any $\alpha$ of the form \eqref{alpha_in_blocks} and any $n\in\mathbb{N}$ such that $a_n(\alpha)<5$, the inequality $\lambda_{n}(\alpha)<\gamma-\varepsilon$ holds.
\end{lem}
\begin{proof}
The proof is similar to the proof of Lemma \ref{lem_k_i}. The only difference is the case when $a_n=4$. One can easily see that due to the definitions of the sequences $b_i$ and $c_i$ and the number $\alpha$, if $a_n=4$, then at least one of the numbers $a_{n-1}$ and $a_{n+1}$ is greater than $2$. Hence,
\begin{equation*}
\lambda_n(\alpha)<4+\frac{1}{3}+1=5+\frac{1}{3}<10-\sqrt{21}\le\gamma.    
\end{equation*}
\end{proof}
In the case $10-\sqrt{21}<\gamma\le 6$,  we will denote by $k_i$ the sequence of indices of $5$-s in the blocks $B^{n_k}_{m_k}$. We formulate the following analogue of Lemma \ref{leftineq}.
\begin{lem}
\label{leftineq_2}
Given $10-\sqrt{21}<\gamma\le 6$ and $\alpha$ of the form \eqref{alpha_in_blocks}. For each $i\in\mathbb{N}$ one has $\alpha_{k_i-1}^{*}>\mu$.
\end{lem}
\begin{proof}
 Arguing as in Lemma \ref{leftineq}, we see that the first $m_i$ partial quotients of $\alpha_{k_i-1}^{*}$ and $\mu$ coincide. As $m_i$ is even, we obtain
\begin{equation*}
\begin{split}
\mu=[0;b_1,\ldots,b_{m_i},\ldots]\le& [0;b_1,\ldots,b_{m_i},1,3,\ldots]\\<&
[0;b_1,\ldots,b_{m_i},1,4,\ldots]=\alpha_{k_i-1}^{*}.
\end{split}
\end{equation*}
\end{proof}
Now, the rest of the proof is identical to the previous case, we replace Lemma \ref{lem_k_i} by Lemma \ref{lem_k_i_2} and Lemma \ref{leftineq} is replaced by Lemma \ref{leftineq_2}. This finishes the proof for the case $10-\sqrt{21}<\gamma\le 6$.

We consider the remaining case $\gamma>6$. As
\begin{equation*}
5+[0;5,\overline{1,4}]+[0;\overline{1,4}]=6,    
\end{equation*}
there exists $\varepsilon=\varepsilon(\gamma)>0$ and a constant $n_0\in\mathbb{N}$ such that if
\begin{equation*}
\alpha=[a_0,a_1,\ldots,\underbrace{4,1,4,1,\ldots,4,1}_{n_0 \text{ pairs}},5^*,5,\underbrace{1,4,1,4,\ldots,1,4}_{n_0 \text{ pairs}},\ldots]    
\end{equation*}
for some $\alpha\in\mathbb{R}\setminus\mathbb{Q}$, one has 
\begin{equation*}
\max\left(\lambda_r(\alpha),\lambda_{r+1}(\alpha)\right)<\gamma-\varepsilon.
\end{equation*}
Once again, the star denotes the $r$-th partial quotient. Let $C$ be the sequence
\begin{equation}
\label{C3}
C:=\underbrace{4,1,4,1,\ldots,4,1}_{n_0 \text{ pairs}},5,5,\underbrace{1,4,1,4,\ldots,1,4}_{n_0 \text{ pairs}}.
\end{equation}
Applying Corollary \ref{cor_repr6}, we obtain the representation \eqref{repr6} of $\gamma$. Denote
\begin{equation}
\label{B3}
B^{n}_m:=b_m,b_{m-1},\ldots, b_1,c,c_1,c_2,\ldots,c_n.  
\end{equation}
Considering $\alpha$ of the form \eqref{alpha_in_blocks}, where $m_i$ and $n_i$ are odd numbers and the blocks $B^{n_i}_{m_i}$ and $C$ are defined in \eqref{B3} and \eqref{C3}, we apply the same argument as in the case $5<\gamma\le10-\sqrt{21}$. This finishes the proof of Theorem \ref{thm_oneside}.
\end{proof}
Now we prove Theorem \ref{thm_twoside}.
\begin{proof}
Due to Theorem \ref{thm_oneside}, it is sufficient to consider the case $4+[0;\overline{3,1}]+[0;2,\overline{1,3}]\le\gamma\le5$ only. Analogously to Theorem  \ref{thm_oneside}, it is sufficient to prove that
\begin{equation*}
\lambda_{n+1}(\alpha)>\lambda(\alpha)-\frac{\varpi(q_n)}{q^2_n} 
\end{equation*}
for infinitely many $n$ and
\begin{equation*}
\label{lambda_n_lt_2}
\lambda_{n+1}(\alpha)<\lambda(\alpha)+\frac{\varpi(q_n)}{q^2_n} 
\end{equation*}
for all $n$ large enough. Due to Corollary \ref{cor_repr4}, $\gamma$ can be represented as a sum \eqref{repr4}. %We will construct $\alpha$ of the form \eqref{alpha_in_blocks} using the blocks $B^n_m$ defined in \eqref{Bnm_def}
\begin{lem}
\label{lem_k_i_no_C}
Let $n_k$ and $m_k$ be arbitrary sequences of positive integers greater than some absolute constant $N$ and growing to infinity. Then for each $\alpha$ of the form
\begin{equation}
\label{alpha_in_blocks_no_C}
\alpha:=[0;B^{n_1}_{m_1},B^{n_2}_{m_2},\ldots]
\end{equation}
where the block $B^n_m$ is defined in \eqref{Bnm_def}, one has $\lambda(\alpha)=\gamma$. Moreover, if $a_n(\alpha)\ne 4$, then $\lambda_{n}(\alpha)<\gamma-0.001$.
\end{lem}
\begin{proof}
The proof is similar to the proof of Lemma \ref{lem_k_i}. The fact that $\lambda(\alpha)=\gamma$ is obtained using the same argument.
To prove the "Moreover" part, note that
\begin{equation*}
3+[0;1,4,\overline{1,3}]+[0;\overline{1,3}]=4.61861\ldots<4+[0;\overline{3,1}]+[0;2,\overline{1,3}]=4.62202\ldots.
\end{equation*}
Hence, if $a_n=3$, we put $N$ be such a number that
\begin{equation*}
\lambda_n(\alpha)<3+[0;1,4,\underbrace{1,3,1,3,\ldots,1,3}_{N \text{ pairs}},\ldots]+[0;\underbrace{1,3,1,3,\ldots,1,3}_{N \text{ pairs}},\ldots]< 3 + [0;1,4,\overline{1,3}]+[0;\overline{1,3}]+0.001
\end{equation*}
for any choice of remaining digits in both continued fractions. Hence, $\lambda_n(\alpha)<\gamma-0.001$. Taking into account that for $a_n\le 2$, we have a trivial estimate
\begin{equation*}
\lambda_n(\alpha)<2+[0;1]+[0;1]=4,
\end{equation*}
the proof is completed.
\end{proof}
In order to complete the proof of Theorem \ref{thm_twoside}, it is sufficient to show that given $\alpha$ of the form \eqref{alpha_in_blocks_no_C}, one has
\begin{equation}
\label{ineq_mod_omega}
\left|\lambda_{k_i}(\alpha)-\lambda(\alpha)\right|<\frac{\varpi(q_{k_i-1})}{q^2_{k_i-1}}.
\end{equation}
In this cases, $k_i$ is the index of $4$ from the block $B^{n_i}_{m_i}$. Using the same argument as in the proof of Theorem \ref{thm_oneside}, we choose $m_i$ and $n_i$ such that
\begin{equation*}
\left|\alpha_{k_i-1}^{*}-\mu\right|<\frac{\varpi(q_{k_i-1})}{2q^2_{k_i-1}}
\end{equation*}
and
\begin{equation*}
\left|\nu + 4-\alpha_{k_i}\right|<\left|\alpha_{k_i-1}^{*}-\mu\right|<\frac{\varpi(q_{k_i-1})}{2q^2_{k_i-1}}.
\end{equation*}
As 
\begin{equation*}
\lambda(\alpha)=\gamma=4+\mu+\nu    
\end{equation*}
and 
\begin{equation*}
\lambda_{k_i}(\alpha)=\alpha_{k_i-1}^{*}+\alpha_{k_i},   
\end{equation*}
the property \eqref{ineq_mod_omega} is satisfied, and thus Theorem \ref{thm_twoside} is proved.
\end{proof}

\subsection*{Acknowledgements}
The authors wish to thank the referees for the very careful reading of the manuscript and for the many valuable comments.

\end{document}